\documentclass[11pt]{article}
\usepackage[utf8]{inputenc}
\usepackage[T1]{fontenc}
\usepackage{amsmath,amssymb,amsthm}
\usepackage{empheq}
\usepackage{graphicx}
\usepackage[hidelinks]{hyperref}
\usepackage{geometry}
\usepackage{float}
\usepackage{abstract}
\usepackage{indentfirst}
\usepackage[cal=boondoxo]{mathalpha}
\newtheorem{theorem}{Theorem}[section]

\newtheorem{definition}[theorem]{Definition}
\newtheorem{assumption}[theorem]{Assumption}
\newtheorem{hypothesis}[theorem]{Hypothesis}
\theoremstyle{definition}
\newtheorem{notation}[theorem]{Notation}
\theoremstyle{remark}
\newtheorem{remark}[theorem]{Remark}

\title{\Large\textbf{On the Spread of a Virus and Traveling Waves, Part I}}
\author{Yizhou~Wang$^{1}$\thanks{Yizhou Wang: Methodology and Computational Analysis Lead. Master's Student, Department of Automation for Scientific Research, Faculty of Computational Mathematics and Cybernetics, Lomonosov Moscow State University; email: saipeizyu@gmail.com.}
\quad Akif~Ibragimov$^{2,3}$\thanks{Akif I. Ibragimov: Corresponding author; email: akif@ipng.ru.}\\[0.6em]
\parbox{0.96\textwidth}{\centering\small
$^{1}$Lomonosov Moscow State University, Moscow, Russian Federation\\[0.3em]
$^{2}$Texas Tech University, Lubbock, TX, USA\\[0.3em]
$^{3}$Oil and Gas Research Institute of the Russian Academy of Sciences\\
(OGRI RAS), Moscow, Russian Federation}}
\date{}

\begin{document}
\maketitle

\begin{abstract}
This paper develops a one-dimensional model for a mobile viral component propagating through a spatially distributed susceptible population. Random viral motion produces diffusion, whereas a response to the relative gradient of susceptible density generates directed transport. The susceptible population is stationary on the scales considered and decreases locally through interaction with the virus. An Einstein-type material-balance argument connects the first two moments of random displacement with the macroscopic transport equation. Closure by the logarithmic susceptible-density gradient yields a coupled diffusion-advection system with a conservative viral flux.

The existence of a positive entire traveling wave is proved when directed transport dominates diffusion, and the limiting case of equal transport and diffusion coefficients is constructed separately. The viral component forms a localized single-hump band that vanishes in both far fields, while the susceptible density increases monotonically from a depleted state to a positive background. The traveling coordinate reduces the partial differential equations to an ordinary differential system. Using the susceptible profile as an independent variable produces a first-order linear equation for viral density. Variation of constants gives an explicit relation between the components. A dimensionless transformation then converts the remaining equation into a logistic equation and supplies closed-form profiles. These formulas identify the unique viral maximum and the far-field behavior.

A conservative implicit-explicit finite-difference method provides an independent numerical check. Diffusion is evaluated at the new time level, the gradient-driven flux explicitly at cell interfaces, and susceptible depletion pointwise. Exact traveling-wave values specify the initial data and time-dependent Dirichlet conditions at both endpoints. Thus, the computation tests preservation and translation of the analytical profiles. A calculation with eight hundred spatial subintervals reproduces the localized viral band and the monotone susceptible front. The largest uniform errors are approximately $1.70\cdot10^{-2}$ for viral density and $2.22\cdot10^{-3}$ for susceptible density. The numerical results corroborate the analytical construction and show that the conservative discretization captures the proposed propagation mechanism.
\end{abstract}


\section{Introduction}

Spatially homogeneous compartmental models describe transitions between population groups, while a spatial description is needed to resolve pathogen transport through a nonuniform susceptible population. A reduced transport model can isolate the mechanisms underlying a moving concentration band and the depletion of the population through which it passes. We consider two fields: the density of a mobile viral component and the density of a susceptible population that is stationary on the scales of interest. Their coupling describes diffusion, gradient-driven transport, and local susceptible depletion. The purpose is to analyze this propagation mechanism within a mathematically tractable biomedical setting.

The relation between microscopic random motion and macroscopic diffusion originates in Einstein's description of Brownian motion~\cite{Einstein05}. Displacement-moment formulations are discussed in~\cite{skor,Brownian Motion}, and gradient-dependent macromolecular transport in~\cite{Schurr_2013}. Diffusion equations with drift and associated asymptotic estimates are studied in~\cite{LuanAkif1}. Einstein-type material-balance models are considered in~\cite{ISH}, while the relation between compartmental infection models and chemotactic systems is developed in~\cite{Yizhou}. These works motivate a formulation in which displacement moments determine the macroscopic transport coefficients and a gradient closure specifies the directed flux.

The traveling-band mechanism has a classical antecedent in the analysis of chemotactic bacterial bands by Keller and Segel~\cite{keller-segel}. More directly, Hevage and Ibragimov~\cite{HevageIbragimov2026} give an explicit traveling-wave analysis for the similar diffusion-advection and stationary-depletion geological structure. 

The present study places this structure in a biomedical interpretation and provides an exact-solution numerical benchmark. The analytical profiles are used with explicit attribution to that prior work.

We give the complete material-balance derivation, traveling-wave construction, and parameter correspondence so that the continuum model and the computational problem can be read together. The proofs retain both the supercritical and critical regimes, including the far-field limits and the unique viral maximum. The numerical study is restricted to the supercritical wave with exact initial data and compatible time-dependent Dirichlet values at both endpoints. A conservative implicit-explicit discretization is specified together with its grid, time-step selection, regularization, and error norms. This formulation allows a direct comparison of the computed solution with the known translating profiles.

The paper is organized as follows. Section~\ref{sec:model} introduces the notation and modeling assumptions and derives the diffusion-advection system from the material balances. Section~\ref{sec:waves} constructs the traveling waves and proves that the viral profile has a unique maximum. Section~\ref{sec:numerics} describes the conservative finite-difference method and the complete initial-boundary value problem. Section~\ref{sec:results} compares the numerical solution with the exact profiles. Section~\ref{sec:conclusion} discusses the results and limitations.

\section{Notation and model construction}\label{sec:model}

We restrict attention to a local depletion law proportional to viral density. At a given viral density, the depletion rate has no additional dependence on susceptible density. Together with the assumption of a stationary susceptible population, this specifies the two-component model analyzed below.

\subsection{Notation and geometric assumptions}

\begin{notation}
Let $\mathcal{V}(x,t)\geq0$ denote the density of the mobile viral component and $S(x,t)>0$ the density of the susceptible population. The spatial coordinate $x\in\mathbb{R}$ parameterizes the propagation route, $t\geq0$ denotes time, $\tau>0$ is a small time increment, and $\eta\in\mathbb{R}$ denotes the spatial shift in the material-balance kernel over that increment.
\end{notation}

\begin{assumption}\label{ass:one-dimensional-route}
The propagation route is approximated by a straight line parallel to the $x$ axis. On the scales considered, the viral component undergoes random spatial displacements, whereas the susceptible population has no independent spatial motion and changes only through local interaction with the virus.
\end{assumption}

\subsection{Material balances and displacement moments}

Let $\varphi_{\mathcal{V},S}(x,t,\eta)$ be the probability density of the spatial shift $\eta$ in the material-balance formulation. It satisfies the normalization condition
\begin{equation}
\int_{-\infty}^{+\infty}
\varphi_{\mathcal{V},S}(x,t,\eta)\,d\eta=1.
\label{eq:normalization}
\end{equation}
Its first and second moments are
\begin{equation}
\begin{aligned}
m_{\mathcal{V},S}(x,t)
&=\int_{-\infty}^{+\infty}
\eta\,\varphi_{\mathcal{V},S}(x,t,\eta)\,d\eta,\\
\sigma_{\mathcal{V},S}(x,t)
&=\int_{-\infty}^{+\infty}
\eta^2\,\varphi_{\mathcal{V},S}(x,t,\eta)\,d\eta.
\end{aligned}
\label{eq:moments}
\end{equation}

\begin{hypothesis}\label{hyp:viral-material-balance}
The Einstein-type material balance for viral density is postulated in the form
\begin{equation}
\begin{aligned}
\mathcal{V}(x,t+\tau)
&=\int_{-\infty}^{+\infty}
\mathcal{V}(x+\eta,t)\varphi_{\mathcal{V},S}(x,t,\eta)\,d\eta\\
&\quad+\mathcal{V}(x,t)\,\partial_x m_{\mathcal{V},S}(x,t).
\end{aligned}
\label{eq:viral-material-balance}
\end{equation}
\end{hypothesis}

For the stationary susceptible population, the transition density is the Dirac delta distribution:
\begin{equation}
\varphi_{S,\mathcal{V}}(x,t,\eta)=\delta(\eta),
\qquad m_{S,\mathcal{V}}=0,
\qquad \sigma_{S,\mathcal{V}}=0.
\label{eq:stationary-susceptible-moments}
\end{equation}

\begin{hypothesis}\label{hyp:susceptible-balance}
The susceptible population decreases locally at a rate proportional to viral density. Its material balance is
\begin{equation}
S(x,t+\tau)=\int_{-\infty}^{+\infty}
S(x+\eta,t)\delta(\eta)\,d\eta-\tau k\mathcal{V}(x,t),
\qquad k>0.
\label{eq:susceptible-material-balance}
\end{equation}
\end{hypothesis}

Assume that the first two moments are of order $\tau$ and that the higher-order Taylor remainder is $o(\tau)$ as $\tau\to0$. For sufficiently smooth densities, expansion of \eqref{eq:viral-material-balance} and \eqref{eq:susceptible-material-balance} in $\tau$ and $\eta$ gives
\begin{equation}
\begin{aligned}
\tau\mathcal{V}_t
&=\frac{\sigma_{\mathcal{V},S}}{2}\mathcal{V}_{xx}
+\partial_x\!\left(\mathcal{V}m_{\mathcal{V},S}\right)+o(\tau),\\
\tau S_t&=-\tau k\mathcal{V}+o(\tau).
\end{aligned}
\label{eq:preliminary-balance}
\end{equation}
The expansion is a formal continuum argument under the stated moment and smoothness assumptions. It connects the postulated material balances to the differential system used in the subsequent analysis.

\subsection{Gradient closure and the differential system}

\begin{hypothesis}\label{hyp:gradient-closure}
The first moment in the material-balance kernel obeys the relative-gradient closure
\begin{equation}
\frac{m_{\mathcal{V},S}}{\tau}
=-\chi\frac{S_x}{S}=-\chi(\ln S)_x,
\qquad \chi>0.
\label{eq:drift-closure}
\end{equation}
\end{hypothesis}

Define
\begin{equation}
D:=\lim_{\tau\to0}\frac{\sigma_{\mathcal{V},S}}{2\tau}>0.
\label{eq:diffusion-coefficient}
\end{equation}
Taking the formal limit in \eqref{eq:preliminary-balance} yields the system studied here:
\begin{equation}
\boxed{\begin{aligned}
\mathcal{V}_t&=D\mathcal{V}_{xx}
-\partial_x\!\left(\chi\mathcal{V}\frac{S_x}{S}\right),\\
S_t&=-k\mathcal{V},
\end{aligned}}
\qquad D>0,\quad\chi>0,\quad k>0.
\label{eq:model}
\end{equation}
The first equation describes diffusion and directed viral transport, and the second describes local depletion of the susceptible population. For $S>0$, the viral equation can be written in the conservative form
\begin{equation}
\mathcal{V}_t+J_x=0,
\qquad J=-D\mathcal{V}_x+\chi\mathcal{V}(\ln S)_x.
\label{eq:continuous-flux}
\end{equation}
In this convention, the macroscopic drift velocity is $\chi(\ln S)_x=-m_{\mathcal{V},S}/\tau$: the shift in the integral balance and the direction of the resulting flux have opposite signs. This divergent form was used in the article for numerical implementation in the Section \ref{sec:numerics}. The gradient closure represents effective directed transport toward regions of higher susceptible density. It is a phenomenological assumption of the model. The system is considered where $S>0$, so that the logarithmic gradient is defined. The traveling solutions constructed below satisfy this condition at every finite spatial point and time; positivity for general initial data requires additional analysis.

\section{Traveling waves for the model}\label{sec:waves}

\begin{hypothesis}\label{hyp:biological-regime}
The susceptible population is initially distributed along the propagation route and is depleted locally by interaction with the virus. The virus is the mobile component, and the relative strength of directed transport and diffusive spreading is characterized by $\chi/D$.
\end{hypothesis}

\subsection{Definition of a traveling band and the wave system}

We seek a solution of \eqref{eq:model} in the form
\begin{equation}
\mathcal{V}(x,t)=\mathcal{v}(\zeta),
\qquad S(x,t)=s(\zeta),
\qquad \zeta=x-ct,
\qquad c>0.
\label{eq:traveling-ansatz}
\end{equation}

\begin{definition}\label{def:strong-traveling-band}
The pair \eqref{eq:traveling-ansatz} is called a strong traveling band if $\mathcal{v},s\in C^2(\mathbb{R})$, the inequalities
\begin{equation}
\mathcal{v}(\zeta)>0,\qquad s(\zeta)>0
\label{eq:positivity}
\end{equation}
hold at every finite $\zeta$, and the far-field conditions are
\begin{equation}
\begin{aligned}
\lim_{\zeta\to-\infty}\mathcal{v}(\zeta)
&=\lim_{\zeta\to+\infty}\mathcal{v}(\zeta)=0,\\
\lim_{\zeta\to-\infty}s(\zeta)&=0,
\qquad\lim_{\zeta\to+\infty}s(\zeta)=s_+>0.
\end{aligned}
\label{eq:far-fields}
\end{equation}
In addition, $s$ is strictly increasing and $\mathcal{v}$ has a unique maximum.
\end{definition}

Substitution of \eqref{eq:traveling-ansatz} into \eqref{eq:model} gives the ordinary differential system
\begin{equation}
\begin{aligned}
-c\mathcal{v}'&=D\mathcal{v}''
-\chi\left(\mathcal{v}\frac{s'}{s}\right)',\\
-cs'&=-k\mathcal{v},
\end{aligned}
\label{eq:traveling-system}
\end{equation}
where a prime denotes differentiation with respect to $\zeta$. Introduce
\begin{equation}
d:=\frac{\chi}{D},
\qquad a:=\frac{c^2}{k},
\qquad \alpha:=\frac{a}{D}=\frac{c^2}{kD}.
\label{eq:dimensionless-parameters}
\end{equation}

\subsection{Construction in the supercritical regime}

\begin{theorem}\label{thm:traveling-band}
Let $D,\chi,k,c,s_+>0$ and $d=\chi/D>1$. Then system~\eqref{eq:model} has a family of strong traveling bands defined for all $(x,t)\in\mathbb{R}^2$. The profiles satisfy
\begin{equation}
\mathcal{v}(\zeta)=\frac{\alpha}{d-1}s(\zeta)-c_0s^d(\zeta),
\qquad c_0=\frac{\alpha}{(d-1)s_+^{d-1}}>0,
\label{eq:profile-relation}
\end{equation}
and, for every $A>0$, are given by
\begin{align}
s(\zeta)&=s_+\left(1+A e^{-c\zeta/D}\right)^{-1/(d-1)},
\label{eq:exact-s}\\
\mathcal{v}(\zeta)&=\frac{\alpha}{d-1}s(\zeta)
\left[1-\left(\frac{s(\zeta)}{s_+}\right)^{d-1}\right].
\label{eq:exact-v}
\end{align}
The function $s$ increases strictly from $0$ to $s_+$, while $\mathcal{v}$ increases to a unique maximum and then decreases to zero.
\end{theorem}

\begin{proof}
Integrating the first equation of \eqref{eq:traveling-system}, we choose the integration constant to be zero. Its consistency with the far-field behavior will be verified at the end of the construction. Together with the second equation, this gives
\begin{equation}
-c\mathcal{v}=D\mathcal{v}'-\chi\mathcal{v}\frac{s'}{s},
\qquad s'=\frac{k}{c}\mathcal{v}.
\label{eq:first-integral-system}
\end{equation}
On any interval where $\mathcal{v}>0$ and $s>0$, we have $s'>0$ and may regard $\mathcal{v}$ as a function of $s$. The chain rule gives
\begin{equation}
\mathcal{v}'=\frac{d\mathcal{v}}{ds}s'
=\frac{k}{c}\mathcal{v}\frac{d\mathcal{v}}{ds},
\qquad \frac{s'}{s}=\frac{k}{c}\frac{\mathcal{v}}{s}.
\label{eq:chain-rule}
\end{equation}
Substituting \eqref{eq:chain-rule} into the first equality of \eqref{eq:first-integral-system} and dividing by the positive quantity $(k/c)\mathcal{v}$ yields
\begin{equation}
-\frac{c^2}{k}=-\chi\frac{\mathcal{v}}{s}
+D\frac{d\mathcal{v}}{ds}.
\label{eq:profile-before-normalization}
\end{equation}
After division by $D$ and use of \eqref{eq:dimensionless-parameters}, this becomes the first-order linear equation
\begin{equation}
\frac{d\mathcal{v}}{ds}-\frac{d}{s}\mathcal{v}=-\alpha.
\label{eq:linear-profile}
\end{equation}

First solve the associated homogeneous equation. Denote its solution by $\mathcal{v}_{\mathrm{hom}}$, where the subscript stands for ``homogeneous.'' For a nonzero solution, separation of variables and integration give
\begin{equation}
\begin{aligned}
\frac{d\mathcal{v}_{\mathrm{hom}}}{\mathcal{v}_{\mathrm{hom}}}
&=d\frac{ds}{s},\\
\ln|\mathcal{v}_{\mathrm{hom}}|+C_1
&=d(\ln s+C_2),\qquad C_1,C_2\in\mathbb{R}.
\end{aligned}
\label{eq:homogeneous-logarithms}
\end{equation}
Set $C_3:=dC_2-C_1$ and $C_4:=\sigma e^{C_3}$, where $\sigma\in\{-1,1\}$ denotes the sign of the nonzero homogeneous solution. Exponentiation gives $\mathcal{v}_{\mathrm{hom}}=C_4s^d$. The zero solution is included by taking $C_4=0$, so $C_4\in\mathbb{R}$ parameterizes all homogeneous solutions.

Apply variation of constants by replacing the constant $C_4$ with a differentiable function $C_4(s)$:
\begin{equation}
\mathcal{v}(s)=C_4(s)s^d,
\qquad \frac{d\mathcal{v}}{ds}
=C_4'(s)s^d+dC_4(s)s^{d-1}.
\label{eq:variation-ansatz}
\end{equation}
Substitution into \eqref{eq:linear-profile} cancels the terms containing $dC_4(s)s^{d-1}$ and gives
\begin{equation}
C_4'(s)s^d=-\alpha,
\qquad C_4'(s)=-\alpha s^{-d}.
\label{eq:variation-equation}
\end{equation}
Since $d>1$, integration with respect to $s$ yields
\begin{equation}
\begin{aligned}
C_4(s)&=-\alpha\int s^{-d}\,ds+C_5\\
&=-\frac{\alpha}{1-d}s^{1-d}+C_5
=\frac{\alpha}{d-1}s^{1-d}+C_5,
\qquad C_5\in\mathbb{R}.
\end{aligned}
\label{eq:variation-integral}
\end{equation}
Consequently,
\begin{equation}
\mathcal{v}(s)=\frac{\alpha}{d-1}s+C_5s^d.
\label{eq:general-profile-relation}
\end{equation}
The condition $\mathcal{v}\to0$ as $s\to s_+$ determines the constant:
\begin{equation}
C_5=-\frac{\alpha}{(d-1)s_+^{d-1}}<0,
\qquad c_0:=-C_5=\frac{\alpha}{(d-1)s_+^{d-1}}>0.
\label{eq:constant-C5}
\end{equation}
Thus \eqref{eq:profile-relation} and \eqref{eq:exact-v} follow, and $\mathcal{v}>0$ whenever $0<s<s_+$.

Substituting \eqref{eq:exact-v} into the second equality of \eqref{eq:first-integral-system} gives
\begin{equation}
s'=\frac{k\alpha}{c(d-1)}s
\left[1-\left(\frac{s}{s_+}\right)^{d-1}\right].
\label{eq:scalar-s}
\end{equation}
Introduce the dimensionless auxiliary function
\begin{equation}
w=w(\zeta):=\left(\frac{s(\zeta)}{s_+}\right)^{d-1},
\qquad 0<w<1.
\label{eq:w-definition}
\end{equation}
Since $w'/w=(d-1)s'/s$ and $k\alpha/c=c/D$, equation~\eqref{eq:scalar-s} becomes the logistic equation
\begin{equation}
w'=\frac{c}{D}w(1-w).
\label{eq:logistic-w}
\end{equation}
Separating variables and integrating explicitly gives
\begin{equation}
\begin{aligned}
\int\left(\frac1w+\frac1{1-w}\right)dw
&=\frac{c}{D}\int d\zeta,\\
\ln\frac{w}{1-w}&=\frac{c}{D}(\zeta-\zeta_0),
\qquad \zeta_0\in\mathbb{R}.
\end{aligned}
\label{eq:w-integration}
\end{equation}
It follows that
\begin{equation}
w(\zeta)=\left[1+\exp\left(-\frac{c}{D}(\zeta-\zeta_0)\right)\right]^{-1}.
\label{eq:w-solution}
\end{equation}
Taking $A=\exp(c\zeta_0/D)>0$ and returning to the definition of $w$ gives \eqref{eq:exact-s}. Thus, $A$ encodes a spatial translation rather than a change in wave shape.

Formula~\eqref{eq:exact-s} shows that $s$ is smooth, strictly increasing, and satisfies \eqref{eq:far-fields}. Differentiating \eqref{eq:exact-v} with respect to $s$ gives
\begin{equation}
\frac{d\mathcal{v}}{ds}=\frac{\alpha}{d-1}
\left[1-d\left(\frac{s}{s_+}\right)^{d-1}\right].
\label{eq:profile-derivative}
\end{equation}
There is a unique critical point, at which
\begin{equation}
s_*=s_+d^{-1/(d-1)},
\qquad \mathcal{v}_{\max}=\frac{\alpha}{d}s_*.
\label{eq:profile-maximum}
\end{equation}
Because $s'>0$, the profile $\mathcal{v}$ increases up to $s=s_*$ and then decreases. Moreover, $\mathcal{v}\to0$, $\mathcal{v}'\to0$, and $\mathcal{v}s'/s\to0$ in both far fields. These limits confirm the zero integration constant in \eqref{eq:first-integral-system}. Differentiating its first equality recovers \eqref{eq:traveling-system}. Therefore, the constructed pair solves the original system~\eqref{eq:model} for every $(x,t)\in\mathbb{R}^2$ and has all the properties required in the theorem.
\end{proof}

\subsection{Explicit profiles in the material-balance parameters}

The following theorem expresses the same family in the original material-balance parameters and also includes the critical case $d=1$.

\begin{theorem}\label{thm:explicit-physical-parameters}
Let $\tau,\mu,\beta,k,c,s_\infty,C_0>0$, with
\begin{equation}
D=\frac{\mu}{2\tau},
\qquad \chi=\frac{\beta}{\tau},
\qquad d=\frac{\chi}{D}=\frac{2\beta}{\mu}\geq1.
\label{eq:physical-parameter-map}
\end{equation}
Then system~\eqref{eq:model} admits a strong traveling band with profiles
\begin{align}
\mathcal{v}(\zeta)&=C_0s^d(\zeta)
\exp\left(-\frac{2\tau c}{\mu}\zeta\right),
\label{eq:explicit-physical-v}\\
s(\zeta)&=\begin{cases}
\displaystyle\left[
\frac{C_0k\mu(d-1)}{2c^2\tau}
\exp\left(-\frac{2\tau c}{\mu}\zeta\right)
+s_\infty^{1-d}\right]^{-1/(d-1)}, & d>1,\\[1.1em]
\displaystyle s_\infty\exp\left[
-\frac{C_0k\mu}{2c^2\tau}
\exp\left(-\frac{2\tau c}{\mu}\zeta\right)\right], & d=1.
\end{cases}
\label{eq:explicit-physical-s}
\end{align}
In both cases, $s(-\infty)=0$, $s(+\infty)=s_\infty$, and $\mathcal{v}(\pm\infty)=0$.
\end{theorem}

\begin{proof}
For $d>1$, formulas~\eqref{eq:explicit-physical-v}--\eqref{eq:explicit-physical-s} agree with Theorem~\ref{thm:traveling-band} after setting $s_+=s_\infty$ and
\begin{equation}
A=\frac{C_0kD(d-1)s_\infty^{d-1}}{c^2},
\qquad c_0=\frac{\alpha}{(d-1)s_\infty^{d-1}}.
\label{eq:explicit-constant-map}
\end{equation}
Indeed, factoring $s_\infty^{1-d}$ out of the first branch of \eqref{eq:explicit-physical-s} gives \eqref{eq:exact-s}. Substitution into \eqref{eq:exact-v} then gives \eqref{eq:explicit-physical-v}, because $\alpha A/[(d-1)s_\infty^{d-1}]=C_0$.

For $d=1$, equation~\eqref{eq:linear-profile} and its solution satisfying $\mathcal{v}(s_\infty)=0$ are
\begin{equation}
\frac{d\mathcal{v}}{ds}-\frac{1}{s}\mathcal{v}=-\alpha,
\qquad \mathcal{v}(s)=\alpha s\ln\frac{s_\infty}{s}.
\label{eq:critical-profile-relation}
\end{equation}
To obtain the second equality, divide by $s$ and integrate
$d(\mathcal{v}/s)/ds=-\alpha/s$; the far-field condition determines the integration constant. Writing $y(\zeta)=\ln[s_\infty/s(\zeta)]$ in the second equality of \eqref{eq:first-integral-system} gives $y'=-(c/D)y$. Its positive solutions have the form $y=B\exp(-c\zeta/D)$, with $B>0$. Taking $B=C_0kD/c^2$ gives the $d=1$ branch of \eqref{eq:explicit-physical-s}, and substitution into \eqref{eq:critical-profile-relation} gives \eqref{eq:explicit-physical-v}.

Direct differentiation verifies \eqref{eq:first-integral-system} for this critical branch. Here $s'= (c/D)ys>0$ and $\mathcal{v}=\alpha s_\infty y e^{-y}$. As $\zeta$ increases from $-\infty$ to $+\infty$, $y$ decreases from $+\infty$ to $0$, so the viral profile vanishes in both far fields and has its unique maximum at $y=1$. Both profiles are smooth and positive at finite $\zeta$. The stated limits follow from the explicit formulas, and differentiation recovers the original wave system. This proves the theorem in both regimes.
\end{proof}

\begin{remark}
The notation in \eqref{eq:physical-parameter-map} parameterizes the continuum coefficients: $\mu=2D\tau$ and $\beta=\chi\tau$. Thus $\beta$ here is not the ratio $\chi/D$; that ratio is denoted consistently by $d$. The coefficients $D$ and $\chi$ remain fixed in the formal continuum limit. The parameter $C_0$ determines the translation through \eqref{eq:explicit-constant-map}; in particular, $\mathcal{v}(0)s(0)^{-d}=C_0$.
\end{remark}

The preceding construction is consistent with the traveling-wave analysis in~\cite{HevageIbragimov2026} and the related chemotactic-band profiles in~\cite{keller-segel}. It supplies the exact reference solution for the numerical study: once the coefficients, positive background, wave speed, and translation constant are fixed, both initial profiles and all endpoint traces are determined.

\section{Conservative finite-difference scheme}\label{sec:numerics}

\subsection{Space and time grids}

For numerical verification, use the conservative form~\eqref{eq:continuous-flux} on the finite interval $[-L,L]$. Introduce uniform grids
\begin{equation}
\begin{aligned}
x_i&=-L+i\Delta x,\quad i=0,\ldots,N,
\qquad\Delta x=\frac{2L}{N},\\
t_n&=n\Delta t,\quad n=0,\ldots,M,
\qquad\Delta t=\frac{T}{M}.
\end{aligned}
\label{eq:grids}
\end{equation}
Let $\mathcal{V}_i^n\approx\mathcal{V}(x_i,t_n)$ and $S_i^n\approx S(x_i,t_n)$. To evaluate the logarithmic gradient near the depleted far field, set
\begin{equation}
q_i^n=\ln(S_i^n+\xi),\qquad\xi>0.
\label{eq:regularized-log}
\end{equation}

\subsection{Conservative implicit-explicit approximation}

At the cell interface $x_{i+1/2}$, approximate the directed part of the flux by
\begin{equation}
F_{i+1/2}^n=\chi\frac{\mathcal{V}_i^n+\mathcal{V}_{i+1}^n}{2}
\frac{q_{i+1}^n-q_i^n}{\Delta x}.
\label{eq:discrete-drift-flux}
\end{equation}
The diffusion term is evaluated at the new time level and the directed flux at the previous level. The resulting implicit-explicit (IMEX) scheme at interior nodes is
\begin{equation}
\begin{aligned}
\frac{\mathcal{V}_i^{n+1}-\mathcal{V}_i^n}{\Delta t}
&=D\frac{\mathcal{V}_{i+1}^{n+1}-2\mathcal{V}_i^{n+1}
+\mathcal{V}_{i-1}^{n+1}}{\Delta x^2}\\
&\quad-\frac{F_{i+1/2}^n-F_{i-1/2}^n}{\Delta x},
\qquad i=1,\ldots,N-1.
\end{aligned}
\label{eq:imex-scheme}
\end{equation}
Writing $r=D\Delta t/\Delta x^2$, equation~\eqref{eq:imex-scheme} is equivalent to the tridiagonal system
\begin{equation}
\begin{aligned}
-r\mathcal{V}_{i-1}^{n+1}+(1+2r)\mathcal{V}_i^{n+1}
-r\mathcal{V}_{i+1}^{n+1}
&=\mathcal{V}_i^n\\
&\quad-\frac{\Delta t}{\Delta x}
\left(F_{i+1/2}^n-F_{i-1/2}^n\right).
\end{aligned}
\label{eq:tridiagonal-system}
\end{equation}
The prescribed endpoint values at $t_{n+1}$ are transferred to the right-hand side in the first and last interior equations. The system is solved by the Thomas algorithm. Susceptible density is then updated locally:
\begin{equation}
S_i^{n+1}=S_i^n-k\Delta t\,\mathcal{V}_i^{n+1}.
\label{eq:S-update}
\end{equation}
The implementation applies a lower bound $\xi$ to this provisional update at interior nodes, namely $S_i^{n+1}\leftarrow\max\{S_i^{n+1},\xi\}$, and then sets the endpoint values from the exact solution. This safeguard affects only extremely small susceptible values in the depleted far field in the present experiment. It is a numerical regularization of \eqref{eq:S-update}, not an additional term in the continuous model. Conservative flux differences in the viral update ensure that interior interface contributions cancel when summed over the grid.

\subsection{Initial and boundary conditions}

The computation uses the material-balance parameters
\begin{equation}
C_0=s_\infty=\tau=\mu=k=1,
\qquad\beta=c=2,
\qquad L=50,\qquad T=2.
\label{eq:numerical-parameters}
\end{equation}
Equations~\eqref{eq:physical-parameter-map} and \eqref{eq:explicit-constant-map} give
\begin{equation}
D=\frac12,\qquad\chi=2,\qquad d=\frac{\chi}{D}=4,
\qquad\alpha=\frac{c^2}{kD}=8,\qquad A=\frac38.
\label{eq:derived-numerical-parameters}
\end{equation}
Thus the exact profiles at $t=0$ are
\begin{equation}
\mathcal{v}_0(x)=\frac{e^{-4x}}
{\left(1+\frac38e^{-4x}\right)^{4/3}},
\qquad s_0(x)=\frac{1}
{\left(1+\frac38e^{-4x}\right)^{1/3}},
\label{eq:initial-profiles}
\end{equation}
and the exact solution is
\begin{equation}
\mathcal{V}_{\mathrm{ex}}(x,t)=\mathcal{v}_0(x-2t),
\qquad S_{\mathrm{ex}}(x,t)=s_0(x-2t).
\label{eq:exact-numerical-solution}
\end{equation}
The initial and boundary data are fully consistent with this solution. We solve the complete initial-boundary value problem
\begin{equation}
\boxed{\begin{aligned}
\mathcal{V}_t&=\frac12\mathcal{V}_{xx}
-2\partial_x\left(\mathcal{V}\frac{S_x}{S}\right),
&&-50<x<50,\quad0<t\le2,\\
S_t&=-\mathcal{V},&&-50<x<50,\quad0<t\le2,\\
\mathcal{V}(x,0)&=\mathcal{v}_0(x),\quad S(x,0)=s_0(x),
&&-50\le x\le50,\\
\mathcal{V}(-50,t)&=\mathcal{v}_0(-50-2t),
&&0<t\le2,\\
S(-50,t)&=s_0(-50-2t),&&0<t\le2,\\
\mathcal{V}(50,t)&=\mathcal{v}_0(50-2t),&&0<t\le2,\\
S(50,t)&=s_0(50-2t),&&0<t\le2.
\end{aligned}}
\label{eq:complete-ibvp}
\end{equation}
Initial conditions prescribe the profiles throughout the spatial interval at $t=0$, whereas Dirichlet conditions prescribe time-dependent values at both endpoints. Because $S$ has no spatial evolution term, its endpoint traces are supplied consistently with its local depletion law and the exact viral data.

The calculation uses $N=800$, $\Delta x=0.125$, $M=854$, $\Delta t=2/854\approx2.342\times10^{-3}$, and $\xi=10^{-30}$. In the code, the step is selected from the initial discrete drift bound: with $b_{i+1/2}^{0}=\chi(q_{i+1}^0-q_i^0)/\Delta x$, a provisional step is $0.05\Delta x/\max_i|b_{i+1/2}^{0}|$, followed by $M=\lceil T/\Delta t_{\mathrm{provisional}}\rceil$ and $\Delta t=T/M$. This describes the time-step selection used in the experiment; it is not a proof of unconditional stability for the centered explicit flux.

\subsection{Error measures}

Comparison with the exact solution uses the discrete norms
\begin{equation}
\begin{aligned}
E_{\mathcal{V}}^{\infty}(t_n)
&=\max_{0\le i\le N}
\left|\mathcal{V}_i^n-\mathcal{V}_{\mathrm{ex}}(x_i,t_n)\right|,\\
E_S^{\infty}(t_n)
&=\max_{0\le i\le N}
\left|S_i^n-S_{\mathrm{ex}}(x_i,t_n)\right|,
\end{aligned}
\label{eq:max-errors}
\end{equation}
and
\begin{equation}
\begin{aligned}
E_{\mathcal{V}}^{2}(t_n)
&=\left(\Delta x\sum_{i=0}^{N}
\left|\mathcal{V}_i^n-\mathcal{V}_{\mathrm{ex}}(x_i,t_n)\right|^2\right)^{1/2},\\
E_S^{2}(t_n)
&=\left(\Delta x\sum_{i=0}^{N}
\left|S_i^n-S_{\mathrm{ex}}(x_i,t_n)\right|^2\right)^{1/2}.
\end{aligned}
\label{eq:l2-errors}
\end{equation}
The superscript $2$ denotes the discrete $L^2$ norm, not the square of an error.

\section{Numerical verification results}\label{sec:results}

Figures~\ref{fig:snapshots-V} and \ref{fig:snapshots-S} show the numerical and exact solutions at $t=0$, approximately $0.501$, $1$, and $2$. Solid curves represent the numerical solution and dashed curves represent \eqref{eq:exact-numerical-solution}; the legend identifies the time associated with each curve. The profiles nearly coincide at the scale of the figures. The viral component retains its localized single-hump profile, the susceptible component remains monotone, and both profiles move to the right with the prescribed speed $c=2$. The plots show the full computational interval $[-50,50]$.

\begin{figure}[htbp]
\centering
\includegraphics[width=0.90\textwidth]{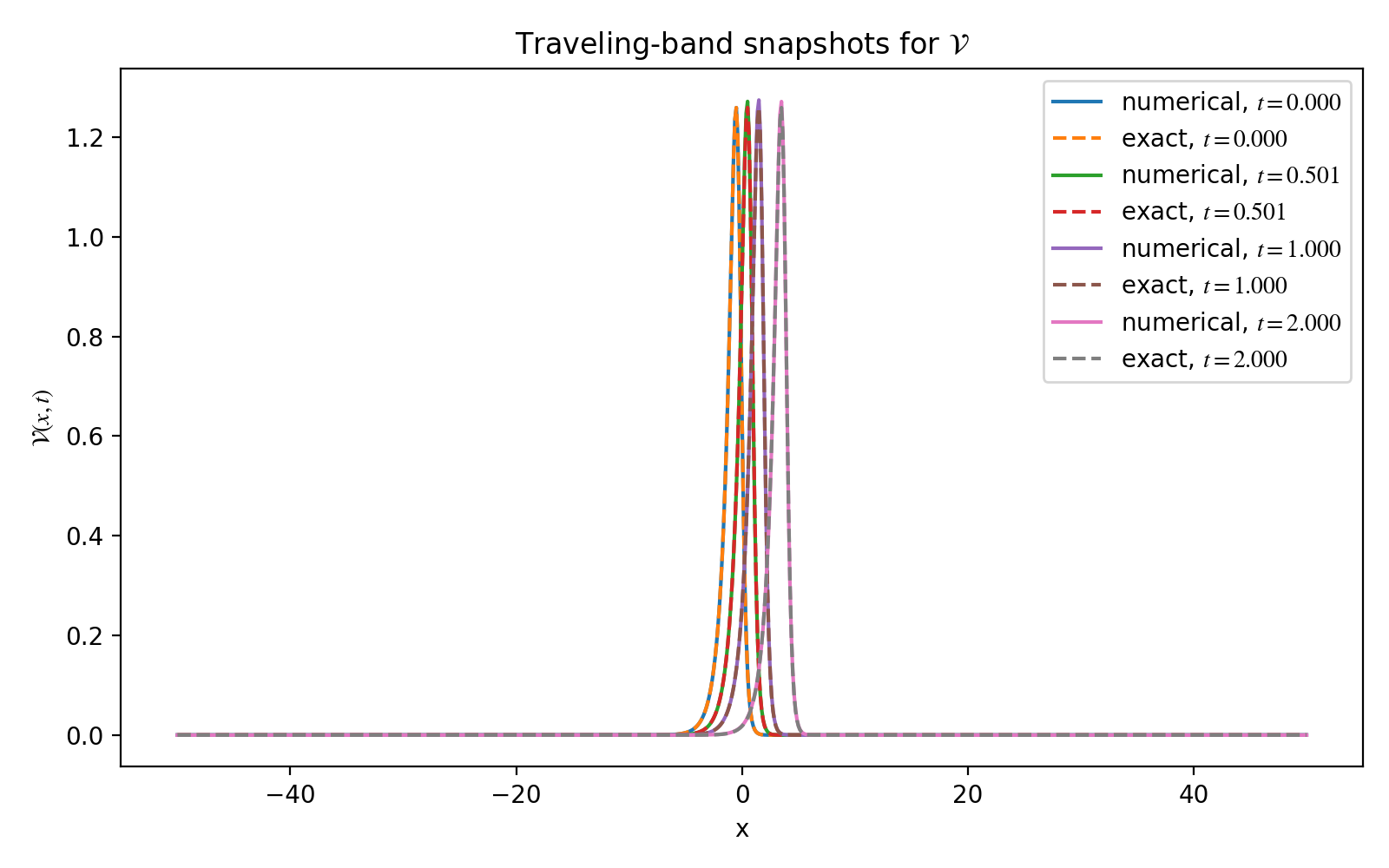}
\caption{Viral-density profiles $\mathcal{V}(x,t)$ at $t=0$, approximately $0.501$, $1$, and $2$ for problem~\eqref{eq:complete-ibvp} on $[-50,50]$. Solid curves show the conservative IMEX solution on the grid with $N=800$; dashed curves show the exact traveling band. The legend identifies the numerical and exact curves at each time. The peak moves to the right with speed $c=2$ while its localized shape is retained.}
\label{fig:snapshots-V}
\end{figure}

\begin{figure}[htbp]
\centering
\includegraphics[width=0.90\textwidth]{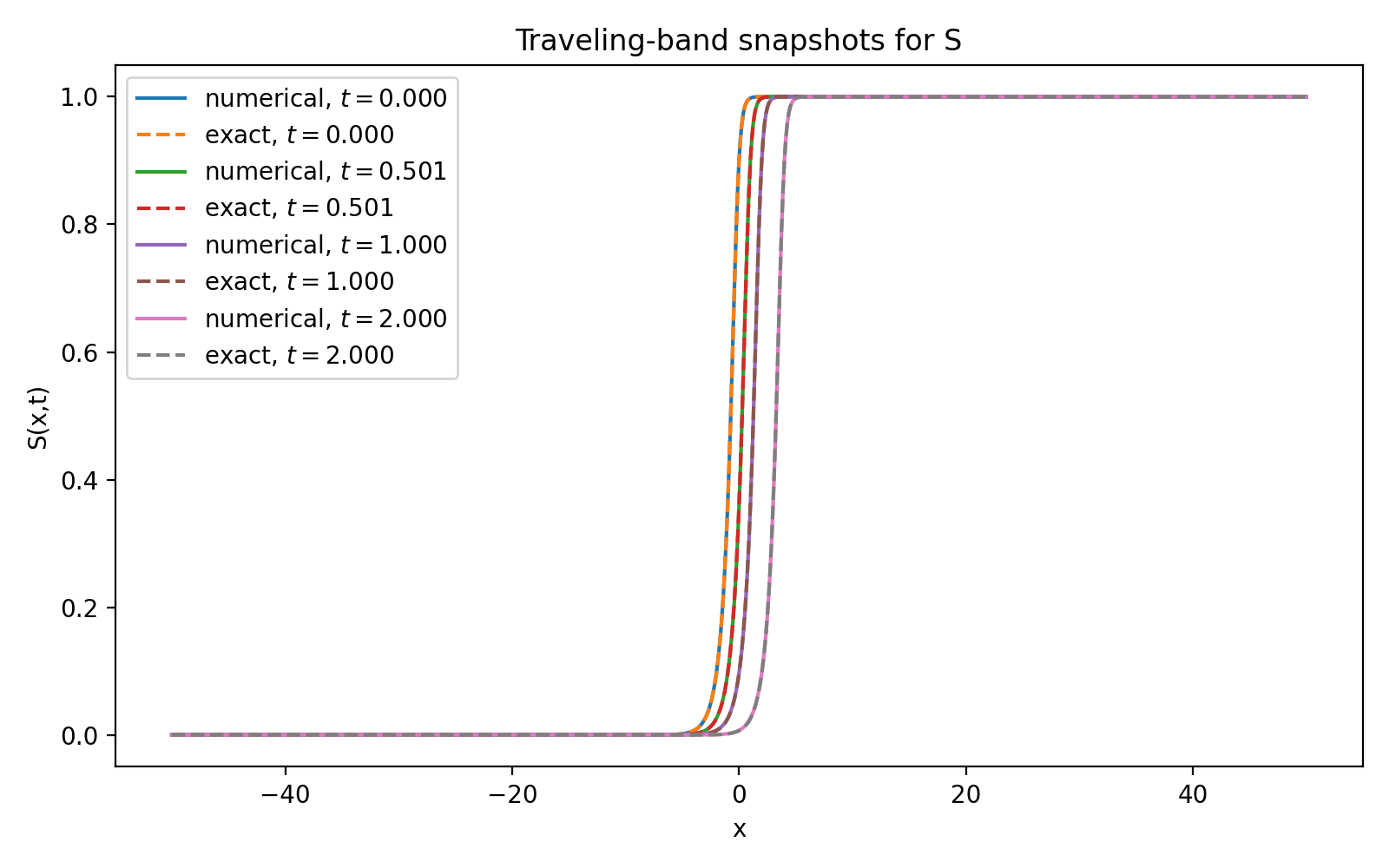}
\caption{Susceptible-population profiles $S(x,t)$ at $t=0$, approximately $0.501$, $1$, and $2$ for problem~\eqref{eq:complete-ibvp} on $[-50,50]$. Solid curves show the numerical solution and dashed curves the exact solution, with the times specified in the legend. The monotone transition from the nearly depleted left state to the right background $S=1$ moves to the right without appreciable distortion.}
\label{fig:snapshots-S}
\end{figure}

Figures~\ref{fig:error-V} and \ref{fig:error-S} show the time histories of the errors in \eqref{eq:max-errors} and \eqref{eq:l2-errors}. Over $0\le t\le2$, their maxima are
\begin{align}
\max_n E_{\mathcal{V}}^{\infty}(t_n)&=1.6985\times10^{-2},
\qquad\max_n E_{\mathcal{V}}^{2}(t_n)=1.2743\times10^{-2},
\label{eq:V-error-values}\\
\max_n E_S^{\infty}(t_n)&=2.2239\times10^{-3},
\qquad\max_n E_S^{2}(t_n)=2.0240\times10^{-3}.
\label{eq:S-error-values}
\end{align}
The mild oscillations in the uniform viral-density error are consistent with the motion of a narrow peak relative to the fixed grid nodes. No sustained error growth is observed in this calculation: after reaching their maxima, the errors remain bounded and decrease toward the end of the run. This is an observation for the stated grid, parameters, and time interval, rather than a general stability or convergence result.

\begin{figure}[htbp]
\centering
\includegraphics[width=0.85\textwidth]{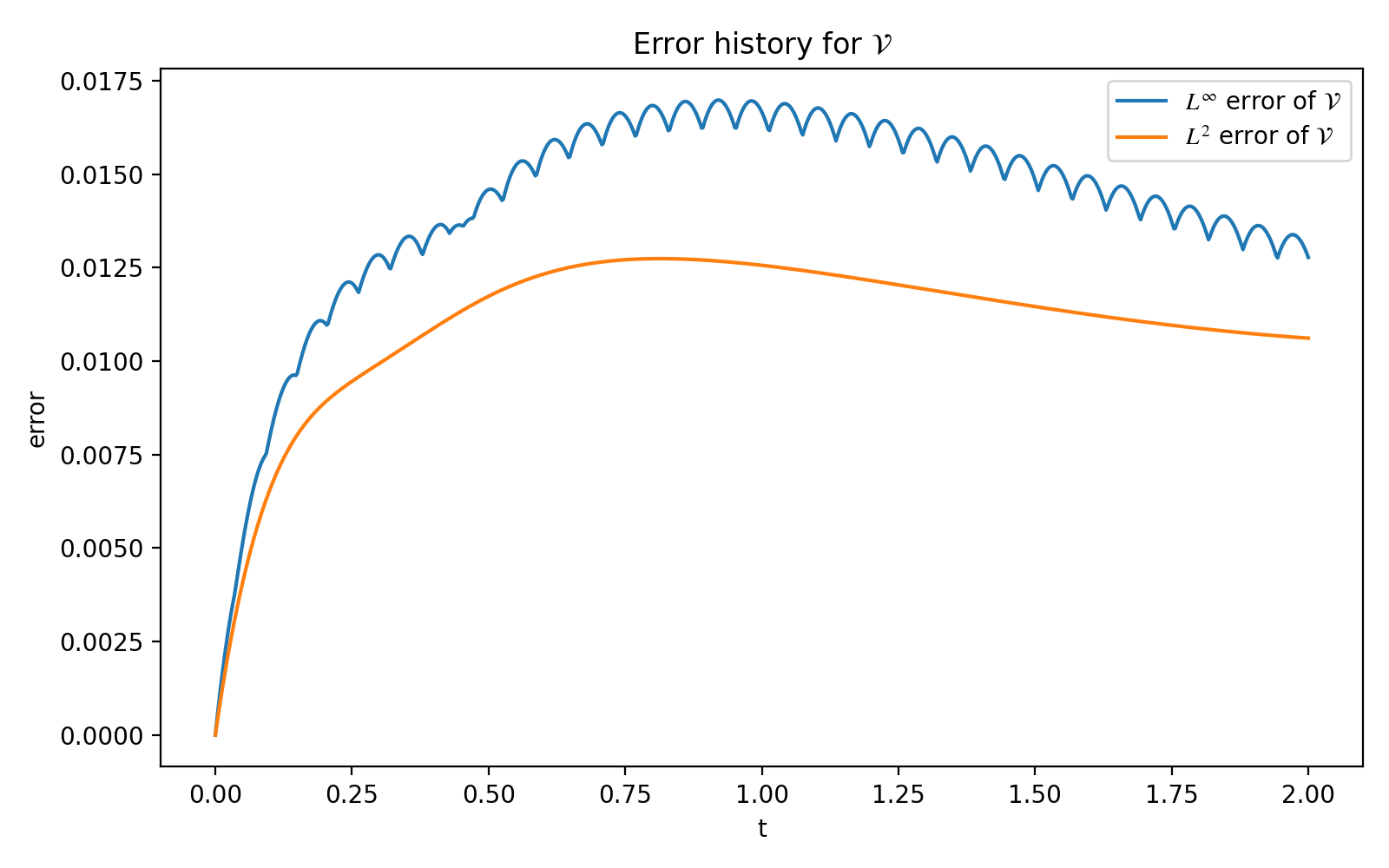}
\caption{Discrete viral-density errors versus time on the grid with $N=800$. The blue curve is $E_{\mathcal{V}}^{\infty}$ and the orange curve is $E_{\mathcal{V}}^{2}$. Their maximum values are $1.6985\times10^{-2}$ and $1.2743\times10^{-2}$, respectively. Both are evaluated against the exact solution~\eqref{eq:exact-numerical-solution}.}
\label{fig:error-V}
\end{figure}

\begin{figure}[htbp]
\centering
\includegraphics[width=0.85\textwidth]{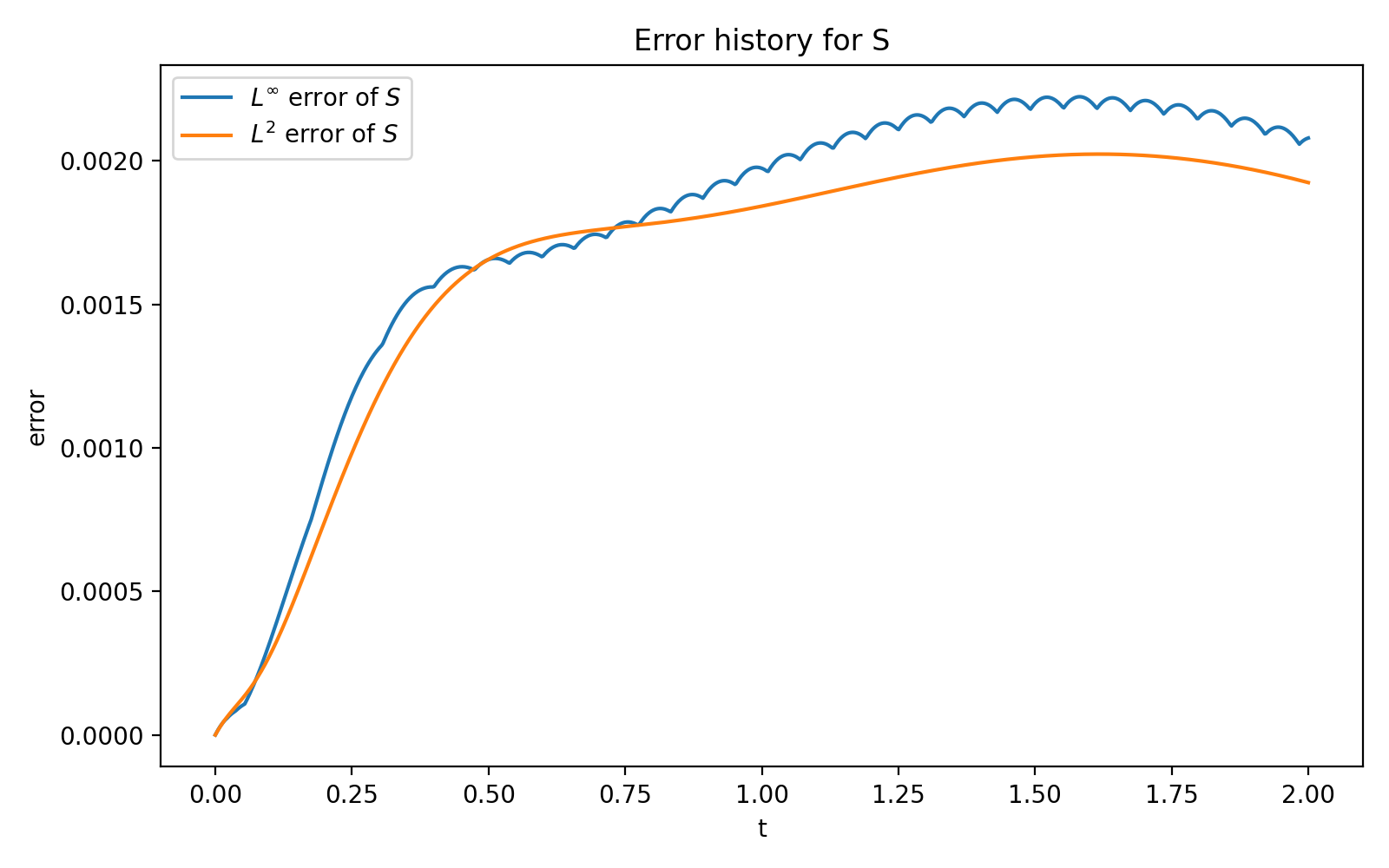}
\caption{Discrete susceptible-density errors versus time on the grid with $N=800$. The blue curve is $E_S^{\infty}$ and the orange curve is $E_S^{2}$. Their maximum values are $2.2239\times10^{-3}$ and $2.0240\times10^{-3}$, respectively. The initial and endpoint values are taken from the same exact traveling wave.}
\label{fig:error-S}
\end{figure}

The analytical proof establishes the exact traveling wave, while the computation checks its reproduction by the discrete method with fully Dirichlet data. The agreement of the profiles and the reported error norms support the consistency of the formulas, endpoint traces, and implementation for this benchmark. They quantify the approximation error over the stated time interval. Biological validation, perturbation stability, and a numerical convergence rate require different evidence and are outside the scope of this calculation.

\section{Conclusion}\label{sec:conclusion}

We have examined a conservative diffusion-advection system that couples mobile viral transport to local depletion of a stationary susceptible population. An Einstein-type material balance and a relative-gradient closure provide the modeling chain. Within this setting, a localized viral band propagates through the susceptible background together with a monotone depletion front.

The complete analytical construction, consistent with~\cite{HevageIbragimov2026}, gives positive entire waves for $\chi/D>1$ and an explicit critical branch for $\chi/D=1$. Viral density has a unique maximum and vanishes in both far fields. Susceptible density connects the depleted state to a positive background. For $\chi/D>1$, reduction to a linear equation for $\mathcal{v}(s)$ followed by a logistic equation for $w$ makes the integration constants and parameter relations explicit. The critical branch follows from the logarithmic profile relation. In each case, the wave speed is prescribed and the positive translation parameter fixes the location of the profiles.

The numerical benchmark links this construction to a conservative IMEX implementation through exact initial profiles and compatible endpoint values. With eight hundred spatial subintervals, the largest uniform errors over $0\le t\le2$ are of order $10^{-2}$ for viral density and $10^{-3}$ for susceptible density. The results demonstrate agreement with the prescribed traveling wave on the reported grid and provide a reference case for assessing implementations of the coupled transport-depletion system.

The biomedical interpretation is mechanistic: the model is not calibrated against medical data and omits susceptible movement, additional clinical compartments, recovery, and viral removal. For more general initial data, the law $S_t=-k\mathcal{V}$ requires particular care because it does not guarantee global positivity of susceptible density. Natural extensions include biologically constrained depletion laws, heterogeneous transport coefficients, and other boundary conditions. Systematic grid refinement and stability analysis would assess the numerical method beyond the exact-solution benchmark presented here.

\section*{Funding}

This work was supported by the Oil and Gas Research Institute of the Russian Academy of Sciences (OGRI RAS), grant No.~122022800272-4.

\section*{Acknowledgments}

The authors thank Isanka Garli Hevage for his contribution to the work of Hevage and Ibragimov~\cite{HevageIbragimov2026}, which provides the theoretical basis for the traveling-wave analysis and proofs presented here. And the authors acknowledge the use of ChatGPT-6.1 Sol  to assist with translating parts of the Russian manuscript into English.

\end{document}